\documentclass{amsart}
\usepackage{amssymb,txfonts,xspace,nicefrac}
\usepackage{graphicx}
\usepackage{tikz}
\usepackage{cancel}
\usepackage{enumerate}

\usetikzlibrary{arrows}
\usetikzlibrary{decorations.pathmorphing}
\usetikzlibrary{patterns}
\usetikzlibrary{shapes}

\newcommand{\Z}{\mathbb Z}
\newcommand{\N}{\mathbb N}
\newcommand{\G}{\mathcal G}

\renewcommand{\H}{{\mathcal G}_{\frac 12}}
\newcommand{\E}{\mathbb E}
\newcommand{\num}{{\sf num}}
\newcommand{\aas}{a.a.s.\ }
\DeclareMathOperator{\length}{len}
\newcommand{\PP}{{\textrm{ \textcircled{\sc i}}}\xspace}
\newcommand{\QQ}{{\textrm{ \textcircled{\sc j}}}\xspace}
\newcommand{\D}{\mathcal{D}}
\newcommand{\sdot}{\! \cdot \!}
\newcommand{\oo}{\infty}
\newcommand{\M}{\mathfrak{m}}

\newtheorem{thm}{Theorem}
\newtheorem{lem}[thm]{Lemma}
\newtheorem{prop}[thm]{Proposition}

\newtheorem{step}{Step}
\newtheorem*{extrastep}{Step 2.5}

\newtheorem{cor}[thm]{Corollary}
\newtheorem{remark}[thm]{Remark}

\newcommand{\alphabetafill}%
{%
\begin{tikzpicture}[yscale=1.3]
\path [fill=blue!10]
  (0,0) rectangle (2.6667,1.5);
\path [fill=red!10]
  (8,0) rectangle (10.667,1.5);
\node (a1) at (1.3333,0.75) {\strut hyperbolic};
\node (a2) at (9.3333,0.75) {\strut trivial};
\node (a3) at (5.3333,0.75) {?};
\draw[->] (-.1,0) -- (10.817,0) node[below] {$\alpha$};
\draw[->] (0,-.1) -- (0,1.65) node[left] {$\beta$};
\node (x0) at (0,-.3) {$0$};
\node (y0) at (-.3,0) {$0$};
\node (y1) at (-.4,1) {$1$};
\node (x1) at (8,-.3) {$1$};
\node (x13) at (2.6667,-.3) {\footnotesize $1/3$};
\node (y23) at (-.5,0.33333) {\footnotesize $1/3$};
\draw (-0.1,1) edge (0.1,1);
\draw (8,-.1) edge (8,.1);
\draw (2.6667,-.1) edge (2.6667,.1);
\draw (-.1,0.33333) edge (.1,0.33333);
\draw [blue,ultra thick] (2.6667,1.5)  -- (2.6667,0.33333);
\draw [dashed, blue, ultra thick]  (2.6667,0)  edge (2.6667,0.33333);
\node [fill=blue, rounded corners] at (2.6667,0.33333) {};
\draw [red,ultra thick] (8,1) -- (8,0);
\draw [dashed, red, ultra thick] (8,1)  edge (8,1.5);
\node [fill=white, draw=red, very thick, rounded corners] at (8,1) {};
\end{tikzpicture}
}

\begin{document}

\begin{abstract}
In the theory of random groups, we consider presentations with any fixed number $m$ of generators
 and many random relators of length $\ell$, sending $\ell\to\infty$.
If $d$ is a ``density" parameter
measuring the rate of exponential growth of the number of relators compared to the length of relators, then
many group-theoretic properties become generically true or generically false at different values of $d$.
The signature theorem for this density model is a phase transition from triviality to hyperbolicity:
for $d<1/2$, random groups are \aas infinite hyperbolic, while for $d>1/2$, random groups
are \aas order one or  two.  We study random groups at the density threshold $d=1/2$.
Kozma had found  that trivial groups are generic for a range of growth rates at $d=1/2$; 
we show that infinite hyperbolic groups are  generic in a different range.
(We include an exposition of Kozma's previously unpublished argument,
with slightly improved results, for completeness.)
\end{abstract}

\title[A sharper threshold for random groups at density one-half]{A sharper threshold for random groups\\ at density one-half}
\author%
[Duchin Jankiewicz Kilmer  Leli\`evre Mackay S\'anchez]%
{Moon Duchin, Kasia Jankiewicz, Shelby C. Kilmer, \\
Samuel Leli\`evre, John M. Mackay, Andrew P. S\'anchez}
\date{\today}
\maketitle

\section{Introduction}
We will study random groups on $m$ generators,
given by choosing relators of length $\ell$ through
a random process.
For a function $\num:\N \to \N$,
let $\G(m,\ell,\num)$ be the probability space of group presentations with $m$ generators and with
$|R|=\num(\ell)$ relators of length $\ell$ chosen independently and uniformly from the 
$(2m)(2m-1)^{\ell -1}\approx (2m-1)^\ell$ possible freely reduced words
of length $\ell$.
Then the usual {\em density model of random groups} is the special case  $\num(\ell)=(2m-1)^{d\ell}$, and the parameter $0\le d\le 1$ is
called the {\em density}.
We will generalize in a natural way by defining
$$\D:=\frac 1\ell \log_{2m-1}(\num(\ell))$$
and saying that the {\em (generalized) density} is $d=\lim_{\ell\to\infty} \D$.

The foundational theorem in the area of random groups is the result of 
Gromov and Ollivier \cite[Thm 11]{Ollivier05} that $d=1/2$ is
the threshold for a phase transition between hyperbolicity 
and triviality.  To speak more precisely, the theorem is 
that for
any $\num$ with $d>1/2$,
a presentation chosen
uniformly at random from $\G(m,\ell,\num)$ will be isomorphic to $1$ or $\Z/2\Z$ with probability tending to $1$ as $\ell\to \infty$;
on the other hand, if $d(\num)<1/2$, a presentation chosen in the same manner will be an infinite, torsion-free, word-hyperbolic group with probability tending to $1$ as $\ell\to \infty$.
(From now on, we will say that a property of random groups is {\em asymptotically almost sure}
(or a.a.s.) for a certain $m$ and $\num$ if its probability tends to $1$
as $\ell\to\infty$.)

Here, we study the sharpness of this phase transition.
Letting $\D=1/2-f(\ell)$ for $f(\ell) = o(1)$ 
lets us use these functions $f$ to parametrize all
cases with generalized density $1/2$.
For simplicity of notation, where $m$ is understood to be fixed in advance, let us write
$\H(f)=\G\left(m,\ell,\  (2m-1)^{\ell(\frac 12-f(\ell))}\right)$.  Constant values of $f(\ell)$ 
change the density, but in the $f(\ell)\to 0$ case 
we show here that the properties of random groups
in $\H(f)$ will depend on the rate of vanishing.

\begin{thm}\label{bigtheorem}
Consider the density $1/2$ model $\H(f)$ for various $f(\ell)=o(1)$.
$$\begin{cases}
G\in \H(f) \text{ \aas infinite hyperbolic, for }
& f(\ell) \ge 10^5 \sdot {\log^{1/3}(\ell)}/{\ell^{1/3}};\\
G\in \H(f) \text{ \aas $\cong 1$ or $\Z/2\Z$, for }
& f(\ell) \le {\log(\ell)}/{4\ell}-{\log \log (\ell)}/{\ell}.
\end{cases}$$
\end{thm}
Here and in the rest of the paper logarithms are taken base $2m-1$ and a group isomorphic
to $1$ or $\Z/2\Z$ is called ``trivial."  
Theorem~\ref{bigtheorem} is illustrated in Figures \ref{slider} and \ref{alphabeta}.

\begin{figure}[ht]
\begin{tikzpicture}
\draw (-4,0)--(4,0);
\node at (0,0) [above] {$ \D $};
\node at (0,-1.33) {$  f(\ell) $};
\node at (-5,0) {$f(\ell)\to 0$};
\node at (-5,-.33) {slowly};
\node at (5,0) {$f(\ell)\to 0$};
\node at (5,-.33) {fast};

\draw (-1.5,.5)--(-1.5,-1) node [below]
  {$\frac{\log ^{1/3} \ell }{\ell ^ {1/3}}$};
\draw (2.5,.5)--(2.5,-1) node [below] {$\frac 1\ell$};
\draw (1.5,.5)--(1.5,-1) node [below] {$\frac{\log \ell}{4\ell}$};
\draw [blue, line width=10pt,-latex] (-1.5,0)--(-4,0);
\draw [red, line width=10pt,-latex] (1.5,0)--(4,0);

\node at (-1.5,1) [left] {$G$ infinite hyperbolic};
\node at (1.7,1) [right] {$G$ trivial};

\end{tikzpicture}
\caption{We study density $1/2$ by taking
$\num(\ell)=(2m-1)^{\ell(\frac 12 -f(\ell))}$ relators
for various functions $f(\ell)=o(1)$.}
\label{slider}
\end{figure}
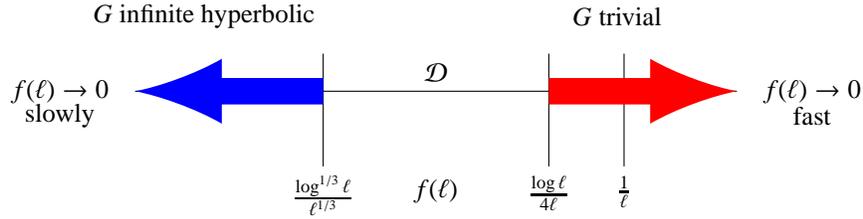

\begin{figure}[ht]
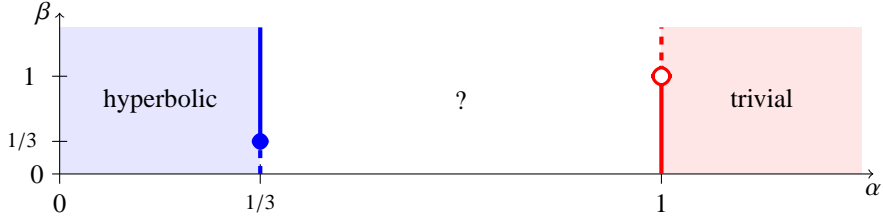

%
% figure: aas properties of G as f = (1/l)^alpha * (log(l))^beta
%
% choose between \alphabetafill and \alphabetahash
%
\alphabetafill
% \alphabetahash
%
\caption{A finer view, taking $f(\ell) = f_{\alpha\beta}(\ell)
= \frac{\log^{\beta} (\ell)}{\ell ^{\alpha} }$.}
\label{alphabeta}
\end{figure}
To interpret
Figure \ref{alphabeta}, note that $\log^\beta (\ell) \ll \ell$ for all $\beta$.
If $G\in\H(f_{\alpha\beta})$ is
\aas trivial, then $G'\in\H(f_{\alpha'\beta'})$ is \aas trivial as well
whenever $\alpha'>\alpha$
or $\alpha'=\alpha$, $\beta'<\beta$.
Similarly if $G\in\H(f_{\alpha\beta})$ is \aas hyperbolic, then the same is true of
 $G'\in\H(f_{\alpha'\beta'})$ whenever $\alpha'<\alpha$ or $\alpha'=\alpha$, $\beta'>\beta$.

This implies in particular 
that setting $d=1/2$ in the classical Gromov
model (which corresponds to $f=0$) 
gives \aas trivial groups.

In unpublished notes from around 2010, Gady Kozma had given an argument for 
triviality at density $1/2$.  We give an expanded exposition here.
By tracking through Kozma's argument as sharply as possible, 
we find triviality at $f(\ell)={\log(\ell)}/{4\ell}-{\log \log (\ell)}/{\ell}$, which corresponds
to any number of relators greater than 
$$
\num(\ell)=(2m-1)^{\ell(\frac 12 - f(\ell))}=(2m-1)^{\frac 12 \ell}\sdot \log(\ell) \sdot \ell^{-1/4}<(2m-1)^{\frac 12 \ell}\sdot  \ell^{-\nicefrac 14+\epsilon}
$$
for any $\epsilon>0$.
(This is slightly sharper than Kozma's conclusion, and he notes that such a result---with a power of $\ell$
factor as we have here---would be interesting.)

On the other hand, our hyperbolicity result applies for any number of relators at most 
$$
\num(\ell)=(2m-1)^{\ell(\frac 12 - 10^5\ell^{-1/3}\log^{1/3}\ell)}= (2m-1)^{\frac 12 \ell}\sdot
(2m-1)^{-10^5 \ell^{2/3}\log^{1/3}(\ell)},
$$
i.e., where $(2m-1)^{\frac 12 \ell}$ is divided by a factor that is intermediate between
polynomial and exponential.
In that case we obtain 
\begin{thm}
For a sufficiently large constant $c$, 
a random group in $\H(f_{\frac 13 \frac 13})$ is \aas $\delta$--hyperbolic with 
$\delta=c\ell^{5/3}$.
\end{thm}

By contrast, for $d<1/2$, the best known hyperbolicity constant is $\delta=c_d\ell$, for 
a coefficient depending on the density.

The proof for triviality given below follows Kozma in using two elementary probabilistic
ingredients:  a ``probabilistic pigeonhole principle" (Lemma~\ref{pigeons}) and a 
``decay of influence estimate" (Lemma~\ref{decay}).  These may be of independent interest, so 
they are formulated in \S\ref{ingredients}
in more generality than we need here.  The main idea is to find a single 
short word that is trivial in $G$ and use it to replace 
the relator set $R$ with an equivalent relator
set $\bar R$ with higher effective density.

For hyperbolicity, we follow Ollivier~\cite[Chapter V]{Ollivier05}
in proving a linear isoperimetric inequality, 
using the local-to-global principle of Gromov as shown by Papasoglu 
to argue that only a limited number 
of Van Kampen diagrams need to be checked, then quoting some classic results of Tutte on enumeration of planar
graphs to accomplish the necessary estimates.

The  sharpest phase transition  
that one could hope for is to have some precise subexponential function $g(\ell)$
and a pair of constants $c_1<c_2$ 
so that $\num(\ell)=c_1 (2m-1)^{\frac 12 \ell} g(\ell)^{-1}$ and $\num(\ell)=c_2 (2m-1)^{\frac 12 \ell} g(\ell)^{-1}$
yield the hyperbolic and trivial cases, respectively.  We hope that in future work we will be able 
to obtain further refined estimates to ``close the gap."

After completing this project we learned of the 2014 preprint \cite{AFL} which 
considers very similar threshold sharpness questions for a different model of random groups,
 called the {\em triangular model}, in which all relators have length three.  
They  find a one-sided threshold for hyperbolicity and show that triviality admits a very sharp phase transition in a sense 
similar to our sense above.
However, hyperbolicity is not known to have such a sharp threshold in either model, and furthermore there is no guarantee that the 
hyperbolicity and triviality thresholds would agree, as we conjecture that they do.

\subsection{Conventions}
We will write $1$ for the group $\{1\}$ and will sometimes
use the term {\em trivial} to mean isomorphic to either
$1$ or $\Z/2\Z$.  Throughout the note, when we show that
groups are \aas ~{\em hyperbolic}, we are proving the same
strong isoperimetric inequality as for the $d<1/2$ 
case, so the groups in our hyperbolic range 
are infinite, and furthermore torsion-free, one-ended,
with Menger curve boundary.  

Since we are concerned with exponential growth with
base $(2m-1)$, %we will write $M=2m-1$ and 
$\log$ will mean $\log_{2m-1}$. %$\log_M$.
% In the following
% section it will be quite notationally convenient to have
% a symbol for its reciprocal, so we will take
% $\M=1/(2m-1)$.

We will use $c, c', c''$
for constants whose values depend on context and  $K, k$
for functions of $\ell$.  As usual,
denote $ f / g \to \infty$ by $f \gg g$.
Write $[n]$ for $\{1,\ldots,n\}$.

For a word $r$ of length $\ell$ we denote by
 $r[i]$ ($1\leq i \leq \ell$) the $i$th letter of $r$,
and write 
$r[i:j]$ (where $1\leq i<j\leq \ell$) for the subword $r[i]\, r[i+1]\cdots r[j]$ of $r$ (so that in particular $r=r[1:\ell]$).
For words $r, r'$ we write
$r=r'$ if $r, r'$ are the same words after free reduction, and 
$r=_G r'$ if $r,r'$ represent the same element of 
group $G$.

As mentioned above, we work with reduced words that need not be
cyclically reduced.  For models of random groups with cyclically
reduced words we expect that the same threshold bounds hold.

%%%
\subsection*{Acknowledgments}
We warmly acknowledge Gady Kozma for ideas and conversations.  The main part of this work 
was conducted during a research cluster supported by NSF CAREER-1255442.

%%%%
\section{Some basic probabilistic facts}\label{ingredients}

\subsection{Distribution of letters in freely
reduced words}

Because the relators in these models of random groups
are chosen by the uniform distribution on freely
reduced words of a given length, it will sometimes
be useful to know the
conditional probability of seeing a particular letter
at a particular position in $r$,
given an earlier letter.

Let $m \ge 2$ be an integer and let $\M = 1/(2m-1)$.

For any positive integer $n$ let $s_n$ be the partial
sum of the alternating geometric series
$$1-\M+\M^2- \dots,$$
i.e., $s_n = \sum_{k=0}^{n-1} (-\M)^{k}$, and $s_0=0$.
Then  $\lim\limits_{n\to\oo} s_n = \frac{1}{1+\M}$.

The following lemma measures the decrease of influence of a
letter on its successors.

\begin{lem}[Distribution of letters]
Consider a random freely reduced infinite word
$w = x_0x_1x_2\dots$
in $m$ generators.
Then for $n > 0$,

$\begin{cases}
\Pr(x_n = x_0)=\M\sdot s_{n-1},& n \text{ even};\\
\Pr(x_n=y)=\M\sdot s_n,& n \text{ even}, y\neq x_0;\\
\Pr(x_n = x_0^{-1})=\M\sdot s_{n-1},& n \text{ odd};\\
\Pr(x_n=y)=\M\sdot s_n,& n \text{ odd}, y\neq x_0^{-1}.
\end{cases}$
\end{lem}

The proof is an easy induction.
Note that as $n \to \oo$ the probability of each
letter appearing at the $n$th place tends to
$\frac{\M}{1+\M}=1/2m$,
recovering the uniform distribution, as one would expect.
We immediately deduce bounds on the conditional
probability of a later letter given an earlier letter.

\begin{cor}[Decay of influence]\label{decay}
For any letters $x,y$ (not necessarily distinct) and for any $n\geq 1$,
$P_n(x,y)=\Pr(x_n=x \mid x_0=y)$ is bounded between
$\M\sdot s_{n-1}$ and $\M\sdot s_n$, i.e.,
$$\begin{array}{lcll}
\M-\M^2+\dots +\M^{n-1}-\M^{n} &\le \ P_n(x,y)\, \le&
\M-\M^2+\dots +\M^{n+1}&(n \text{ even})\\
\M-\M^2+\dots -\M^{n+1} &\le\ P_n(x,y)\, \le &
\M-\M^2+\dots -\M^{n-1}+\M^n &(n \text{ odd}).
\end{array}$$
In particular,
$\frac{2m-2}{(2m-1)^2}\le
\Pr(x_2=x \mid x_0=y)\le \frac{1}{2m-1} .$
\end{cor}

\subsection{A generalized ``probabilistic pigeonhole principle"}

Consider $z$ red balls and $z$ blue balls,
and so on for a total of $q$ colors.  Each of these $qz$ balls is
thrown at random into one of $n$ boxes, giving $[n]$-valued random variables
$x_1,\dots,x_{qz}$.
We bound the probability that there is some box with
balls of all colors.

\begin{lem}[Probabilistic
pigeonhole principle on $q$ colors]
\label{pigeons}
Let $\mu$ be any probability measure on $[n]$.  Fix arbitrary $q,z\in \N$ such that
$z\ge 2 n^{1-1/q}$.  Then if  
 $x_1, \dots, x_{qz}$ are chosen randomly and independently under $\mu$, 
\[ \Pr(\exists\,  i_1,i_2,\dots,i_q
\text{ with } (j-1)z < i_j \le jz, \ 
x_{i_1} = x_{i_2}= \dots = x_{i_q})
\geq 1 - e^{\textstyle{-cz/{n}^{1-1/q}}} \]
for any $c\le -\frac 14 \ln(1-2^{-q})$, or in particular
$c\le 2^{-q-2}$.
\end{lem}
Note that as $n\to\infty$ a $q$-color
coincidence is asymptotically almost sure
as long as
$z \gg n^{1-1/q}$, and in particular a 2-color coincidence occurs if $z\gg \sqrt n$.
We further remark that this is equivalent to another probabilistic pigeonhole principle
(that for $z \gg n^{1-1/q}$ uncolored balls in $n$ boxes, some box contains at least $q$ balls
a.a.s.),
in the sense that each applies the other.

\begin{proof}
We start by considering the case of a 3-color coincidence ($q=3$). Let
$$X := \# \{ (i_1, i_2, i_3) \mid
(j-1)z<i_j\le jz, \
x_{i_1} = x_{i_2} = x_{i_3} \}.$$ 
Since $X \geq 0$ we can bound $\Pr(X>0)$ using the classical inequality $\Pr(X>0) > \E^2[X]/\E[X^2]$.
We compute expectation by finding the probability of coincidence for some choice
of distinct $i_1, i_2, i_3$  and multiplying by $z^3$:
$$ \E[X] = z^3 \Pr(x_{i_1} = x_{i_2} = x_{i_3}) = z^3 \sum_{p=1}^n \mu^3 (p). $$
We next write $X=\sum_{i_1}\sum_{i_2}\sum_{i_3} \delta_{x_{i_1}=x_{i_2}=x_{i_3}}$ and reindex as
$X=\sum_{i_4}\sum_{i_5}\sum_{i_6} \delta_{x_{i_4}=x_{i_5}=x_{i_6}}$, so by symmetry we get 
\begin{align*} \E[X^2] = z^3 \Pr(x_{i_1} = x_{i_2} = x_{i_3}) &+ 3z^3 (z-1) \Pr(x_{i_1} = x_{i_2} = x_{i_3} = x_{i_4}) \\
&+ 3z^3(z-1)^2 \Pr(x_{i_1} = x_{i_2} = x_{i_3} = x_{i_4} =x_{i_5})\\
&+ z^3(z-1)^3\Pr(x_{i_1}=x_{i_2}=x_{i_3}, x_{i_4}=x_{i_5}=x_{i_6})
\end{align*}
with respect to any fixed $i_1,\dots,i_6$.

Using $1<r<s \Longrightarrow \| x \|_r \geq \| x \|_s$, 	we get
\[ \E[X^2] \leq z^3 \left (\sum_{p=1}^n \mu^3 (p) \right)^{3/3} + 3 z^4 \left( \sum_{p=1}^n \mu^3 (p) \right)^{4/3} + 3z^5 \left (\sum_{p=1}^n \mu^3 (p) \right)^{5/3} + z^6 \left (\sum_{p=1}^n \mu^3 (p) \right)^{6/3}.  \]

These expectation formulas easily generalize from 3 to any number $q$  of colors:
\[ \E[X] = z^q \sum_{p=1}^{n} \mu^q(p) \ \ ; \qquad
\E[X^2]\leq\sum_{i=0}^q \left[{q \choose i}\cdot z^{q+i}
\cdot \left( \textstyle{\sum_{p=1}^n \mu^q(p)}
\right)^{\frac{q+i}{q}} \right].  \]

The probability of a coincidence is at least $\E^2[X]/\E[X^2]$.
First let us consider a simple case, where the number
of balls of each color is chosen to get good cancellation:
set $z_0:= \left(\sum_{p=1}^n \mu^q (p) \right)^{-1/q}$, so that $1\le z_0\le n^{1-1/q}$.
Then we get 
$\Pr(X>0 \mid z\ge z_0)> 1/2^q$.

The general case is $z=\gamma z_0$ for arbitrary $\gamma$.
Divide up each of the intervals
$\bigl((j-1)z,jz\bigr]$ into subintervals of length $\lceil z_0 \rceil$, with the last
subinterval longer if necessary, and let $\rho$ be the number of subintervals (the hypothesis
that $z\ge 2 n^{1-1/q}$ ensures that $\gamma/4\le \rho\le \gamma$).  
Let $X_k$ count the number of  $q$-color coincidences which occur in the respective $k$th subintervals.  
The above calculation
tells us that $\Pr(X_k>0)>1/2^q$.

By H\"older's inequality,
we have
$$1= \sum_{p=1}^n \mu(p)
= \sum_{p=1}^n \mu(p)\cdot 1
\le \left(\sum_{p=1}^n \mu^q(p)\right)^{1/q}\cdot
n^{1-1/q}$$

Thus we have $\gamma \ge \dfrac{z}{n^{1-1/q}}$.
It follows that
$$
\Pr(X >0)\ge 1-\prod_{k=1}^\rho\left(\Pr(X_k=0)\right)
\ge 1-\left(1-2^{-q} \right)^{\gamma/4}
\ge 1-\left(1-2^{-q} \right)^{\textstyle{ \frac 14 \sdot \frac{z}{n^{1-1/q}}}}.\qedhere
$$
\end{proof}

\noindent We  emphasize that this result does not depend
on the choice of probability distribution $\mu$.

%%%%
\section{The trivial range}

The usual proof that a random group $G$ is trivial at densities $d>1/2$ uses the
probabilistic pigeonhole principle to show that there are pairs of relators $r_1,r_2$
which have different initial letters $r_1[1]=x, r_2[1]=y$, but with
the remainder of the words equal.  
Consequently $r_1r_2^{-1}=xy^{-1}$ is trivial.
In this way one shows that a.a.s.\ all generators and their inverses are equal in $G$.

To show triviality at density $d=1/2$ is more involved.
The overall plan here is to find shorter
trivial words than the ones from relator set $R$; treating these as 
an alternate relator set
will push up the ``effective density" of $G$,
then a similar argument as before will show that the group is trivial.

\begin{thm}[Sufficient conditions for triviality]\label{triviality}
Given any $f(\ell)=o(1)$,
suppose there exists a function
$k: \N \to \N$  with $k(\ell)\le \ell$ for all $\ell$
and such that
\begin{equation}\tag{$\star$}\label{star}
k-2 \ell f  \to \infty
\end{equation}
and
\begin{equation}\tag{$\spadesuit$}\label{spade}
\frac{\ell - 2}{(2k+2)(2m-1)^{2k}} \to\infty
\end{equation}
%\begin{equation}\tag{$\diamondsuit$}
%k\cdot \mu\cdot b - f(\ell) \cdot \ell \xrightarrow{}\infty
%\end{equation}
%where $\mu > 0$,
as $\ell\to\infty$.
Then \aas $G\in \H(f)$ is $1$ or $\Z/2\Z$.
\end{thm}

\begin{cor}\label{triviality-cor}
% $\forall~ 0<2c<c'<\frac 12$, the functions
% $f(\ell)=\frac{c\log(\ell)}{\ell}$ and
% $k(\ell)=c'\log(\ell)$
% satisfy \eqref{star}, \eqref{spade}.
% Thus, for any $0<c<1/4$,
% a random group in 
% $\H \left(\frac{c\log(\ell)}{\ell} \right)$ is
% \aas $1$ or $\Z/2\Z$. 
The functions $k(\ell)=\frac12 \log(\ell)-\log\log(\ell)$ and
$f(\ell) = \frac{\log(\ell)}{4\ell} - \frac{\log\log(\ell)}{\ell}$
satisfy \eqref{star}, \eqref{spade}.
Thus a random group in 
$\H \left( \frac{\log(\ell)}{4\ell} - \frac{\log\log(\ell)}{\ell}\right)$ is
\aas $1$ or $\Z/2\Z$. 
\end{cor}

% \begin{remark} The $f(\ell)=1/\ell$ case is implied by this one, but is also easily verified directly with 
% $k(\ell)=\sqrt{\log(\ell)}$.
% \end{remark}

\subsection*{Outline of the proof of Theorem~\ref{triviality}}
\begin{enumerate}[(Step 1)]
\item Using the pigeonhole principle (Lemma \ref{pigeons}), we find a freely reduced word $w$ of length $2k$ such that $w=_G 1$.
The existence of such a $w$ is guaranteed by \eqref{star},
and we will use it to reduce other relators.
\item In each relator $r$ we set aside the first two letters for later use, and then 
chunk the last $\ell-2$
letters into $b$ blocks of size $(2k+2)(2m-1)^{2k}$,
with the last block possibly smaller.
The \eqref{spade} condition says that $b\to\infty$.   
We show that 
 $w$ appears in one of these blocks
surrounded by non-canceling letters is
 with probability $>\frac 1 4$.

\item With these reductions, the probability that $r$
reduces to length at most $\ell'=\ell-\frac{bk}2$ is 
more than $1/3$.
\item Finally we show that for this choice of $\ell'$,
conditions
\eqref{star} and \eqref{spade} ensure that
$d\ell-\frac{\ell '}2\to\infty$.
From this we deduce that
 for any pair of generators $a_i$, $a_j$,
we can almost surely find two reduced relators
that start with $a_i$, $a_j$, respectively, and match
after that.
Therefore $a_i=_G a_j$ for all pairs of generators
(including $a_j=a_i^{-1}$),
which establishes the trivality result.
\end{enumerate}

\begin{proof}[Proof of Theorem~\ref{triviality}]

\begin{step}
Suppose $k-2\ell f \to \infty$.  Then \aas
there exists a reduced word $w$ of length $2k$ such that $w=_G 1$.
\end{step}

For each $r\in R$ the word $r[k+1:\ell]$ is one of the
$2m(2m-1)^{\ell-k-1}$ reduced words of length $\ell-k$.
We will find two relators $r_1,r_2$ such that their tails
match
(i.e., $r_1[k+1:\ell]=r_2[k+1:\ell]$) but they differ in
the previous letter ($r_1[k]\neq r_2[k]$).
We can conclude that $w=r_1 r_2^{-1}$ reduces to a word
of length $2k$.

For any word $w$ of length $p$, we  define
$R_w$ to be the subset of relators beginning with
that word:
$$R_w:=\{r\in R \mid r[1:p]=w  \}.$$
For letters $x,y,z$, $R_{xz}$ and $R_{yz}$ are disjoint
as long as $x$ and $y$ are distinct and neither one
is equal to $z^{-1}$.  Fix such letters $x,y,z$.
There are $2m(2m-1)$ possible two-letter reduced words and since we choose $R$ uniformly, the 
law of large numbers tells us that \aas
\[
|R_{xz}|>\frac 1{2m(2m-1)+1}\cdot |R|
=\frac {(2m-1)^{\ell(\frac 12 - f(\ell))} }{2m(2m-1)+1}.
\]
The same holds for $R_{yz}$.

We will check that
\[
\frac{(2m-1)^{\ell(\frac 12 - f(\ell))}}{2m(2m-1)+1}
\gg \sqrt{2m(2m-1)^{\ell-k-1}}.
\]
Using $2m-1\ge 3$, 
we have
$2m(2m-1)+1 \le 2(2m-1)^2$ and
$2m(2m-1)^{\ell-k-1} \le 4(2m-1)^{\ell-k}$,
which gives 
\[
\frac{(2m-1)^{\ell(\frac 12 - f(\ell))}}{2m(2m-1)+1}\cdot
\frac 1 {\sqrt{2m(2m-1)^{\ell-k-1}}}
\ge c (2m-1)^{{\ell(\frac 12 - f(\ell))} -\frac{\ell-k}{2}} = c (2m-1)^{\frac k2 - \ell f(\ell)}
\]
where $c=1/4(2m-1)^2>0$.
The right-hand side goes to infinity precisely
when \eqref{star} holds.

The purpose of introducing the letter $z$ is to ensure
that the tails of words in
$R_{xz}$ and $R_{yz}$ have
the same distribution.
Hence we can apply Lemma \ref{pigeons} (with $q=2$) to conclude that
\aas there exist $r_1\in R_{xz}$ and
$r_2\in R_{yz}$
such that $r_1[k+1:\ell]=r_2[k+1:\ell]$. Then setting
$w=(r_1[1:k])^{-1}\cdot r_2[1:k]$,
we have $w=_G 1$.

\begin{step}
Let $w$ be as above
and  $r$ be  freely reduced of length $\ell$.
Set $s=(2k+2)(2m-1)^{2k}$ and
$b=\lfloor\frac{\ell-2}{s}\rfloor$.
From the third letter on, divide $r$ into $b$ blocks
of length $s$ (with possibly one shorter block at the end).
For each such block $B$, let $\lambda(B)$ be the last
letter of $r$ preceding $B$.
Then
the conditional probability that $w$ appears in $B$
given any particular value of  $\lambda(B)$ is uniformly bounded
away from $0$ as follows:
\[
\forall g, \quad
\Pr\left(w \textrm{ appears in }B \mid \lambda(B)=g \right)
\geq 1-e^{-2/3}.
\]
\end{step}

Write $w=w_1\cdots w_{2k}$,
let $B$ be a block of size $(2k+2)(2m-1)^{2k}$, and divide it 
into $(2m-1)^{2k}$ subblocks $B_1,\dots, B_{(2m-1)^{2k}}$ of size 
$2k+2$. Let $E_i$ be the event that the word $w$ appears as $B_i[2:2k+1]$. See Figure \ref{subblock}.

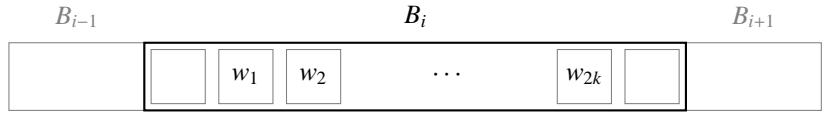
\begin{figure}[ht]
\begin{tikzpicture}[scale=.9]

\draw [gray] (8,0) rectangle node [above=15pt] {$B_{i+1}$} (10,1);
\draw [gray] (-2,0) rectangle node [above=15pt] {$B_{i-1}$} (0,1);
\draw [thick] (0,0) rectangle node [above=15pt] {$B_{i}$} (8,1);
\node at (1.5,.5) {$w_1$};
\node at (2.5,.5) {$w_2$};
\node at (4.5,.5) {$\cdots$};
%\node at (6.5,.5) {$\cdots$};
%\node at (4.5,.5) {$a_k$};
%\node at (5.5,.5) {$b_k^{-1}$};
%\node at (7.5,.5) {$w_{2k-1}$};
\node at (6.5,.5) {$w_{2k}$};
\foreach \x in {0,1,2,6,7}
{\draw [gray] (\x+.1,0.1) rectangle (\x+.9,.9);}
\end{tikzpicture}

\caption{A part of block $B$. \label{subblock}}
\end{figure}

Let us compute the probability of $E_i$ given that none of
$E_1,\cdots E_{i-1}$ happens and given any last letter
$g_0$ before $B_i$. For $1\leq i\leq (2m-1)^{2k}$, we have
\begin{align*}
P_i&=\Pr(E_i \mid \neg E_1, \dots, \neg E_{i-1},
\lambda(B_i)=g_0)\\
&\stackrel{(1)}{=}
\Pr(B_i[2]=w_1 \mid
\neg E_1, \dots, \neg E_{i-1}, \lambda(B_i)=g_0)\cdot \left(\frac{1}{2m-1}\right)^{2k-1}\\
& \stackrel{(2)}{\geq}
\frac{2m-2}{(2m-1)^2}\cdot
\left(\frac{1}{2m-1}\right)^{2k-1}\ge
\frac 23 (2m-1)^{-2k}.
\end{align*}
Equality (1) follows from the fact that only $B_i[2]$ could be affected by previous letters in $r$.
Inequality (2) is an application of the decay of influence estimate (Corollary~\ref{decay}),
which guarantees that $\Pr(x_2=x \mid x_0=y)
\ge (2m-2)/(2m-1)^2$
for any $x,y$.  We deduce that
\[
\prod_{i=1}^{(2m-1)^{2k}} \Pr(\neg E_i\mid \neg E_0, \dots,
\neg E_{i-1}, \lambda(B)=g_0)
=\prod_{i=1}^{(2m-1)^{2k}}
\left(1-P_i\right)
\leq
\left(1-\frac 23 (2m-1)^{-2k}\right)^{\textstyle{(2m-1)^{2k}}}
\leq e^{-2/3},
\]
and so finally for any $g_0$,
\[
\Pr(w \textrm{ appears in }B \mid \lambda(B)=g_0)
\geq 1-e^{-\frac 23}>\frac 14.
\]

\begin{extrastep}\label{extrastep}
If there exists a subword $w'$ of $B$ of the form
\[
w'=sdwd^{-1}t
\]
for any word $d$ and letters $s,t$ with $s\neq t^{-1}$,
then we say that \emph{$B$ has a $w$-reduction}.  (In this case $w=_G 1 \implies w'=_G st$,
and $B$ remains freely reduced.)
For $k$ sufficiently large we bound \[
\Pr(\text{$B$ has a $w$-reduction}
\mid \lambda(B)=g_0)>\frac 14.
\]
\end{extrastep}

We want to bound from above the conditional probability that $w$ appears in $B$ in the wrong form for
a $w$-reduction.
This only happens if $B$ starts or ends with
$d w d^{-1}$
for some word $d=d_1\cdots d_n$.
Let us compute the probability that $B$ starts this way.
First we bound the probability that $w$
appears in the right place, then conditioning on that
we bound the other needed coincidences.
We have $\Pr(B[n+1]=w_1)\le \frac{1}{2m-1}$, and
$$\Pr(B[n+1:n+2k]=w \mid B[n+1]=w_1)
=\frac{1}{(2m-1)^{2k-1}}.$$
Next we consider whether
$B[n+1-j]=B[n+2k+j]^{-1}$ for each $j=1,\dots,n$.
For $j=1$, we have
$$
\Pr(B[n]=B[n+2k+1]^{-1}) = \frac{1}{2m-1}
\quad \text{or}\quad  \frac{2m-2}{(2m-1)^2},
$$
depending on whether $w_1=w_{2k}$ or not, but in either
case this is $\le 1/(2m-1)$.
For $j=2,\dots,n-1$, the conditional probability
is exactly $1/(2m-1)$.  For $j=n$, we have
the same two possibilities as before, depending on
whether $\lambda(B)=d_2$.
So all together we find
$$
\Pr(B \textrm{ starts with }dwd^{-1} \mid \lambda(B)=g_0)
\le \left(\frac{1}{2m-1}\right)^{2k+n}.
$$
The same inequality holds for finding $dwd^{-1}$ at the
end of $B$, so
$$
\Pr(\text{$w$ appears in $B$ with no $w$-reduction})
\le 2 \sum_{n=0}^\infty \left(\frac{1}{2m-1}\right)^{2k+n},
$$
and the right-hand side goes to $0$ as long as
$k\to\infty$.
So finally for sufficiently large $\ell$ (and therefore
$k$),
\[
\Pr( \textrm{$B$ has a $w$-reduction}
\mid \lambda(B)=g_0)>\frac 14.
\]

\begin{step}
For each relator $r$ denote by $\bar r$ the word obtained 
by performing the first available
$w$-reduction in each block. By comparing to an appropriate 
Bernoulli trial, for $k$ sufficiently large we  show that
\[
\Pr\left(\# \{\textrm{reductions of }w\textrm{ in }B\}
>\tfrac b 4   \mid r[1:2]=g_1g_2 \right)
>\frac 1 3,
\]
and conclude that
\[
\Pr\left(\length(\bar r)<\ell -\tfrac{kb}2
\mid r[1:2]=g_1g_2 \right)>\frac 1 3.
\]
\end{step}
Let $X_i$, for $i=1,\dots, b$, be i.i.d.\ random variables such that $X_i=1$ with probability $1/4$
and $X_i=0$ with probability $3/4$. Then by the central limit theorem,
\[
\lim_{b\to\infty}\Pr\left({\textstyle\sum X_i}>\tfrac b4\right)
%=\lim_{b\to\infty}\Pr\left(\sqrt{b}\left(\tfrac{\sum X_i}{b}-p\right)>0\right)=1-\Phi(0)
=\frac 12.
\]
Let $\widetilde X_i$ be the indicator random variable for a $w$-reduction in the $i$th block $B^{(i)}$ of $r$.
The variables $\widetilde X_1,\widetilde X_2, \dots$ are not independent, but each $\widetilde X_i$ depends only on $\lambda(B^{(i)})$. By Step 2.5 we know that for any $g_0$,
\[
\Pr(\widetilde X_i=1 \mid \lambda(B^{(i)})=g_0)>\frac 14 = \Pr(X_i=1),
\]
so
\[
\Pr\left(\textstyle{\sum_{i=1}^b\widetilde X_i} >\tfrac b4 \mid  r[1:2]=g_1g_2\right)
\geq \Pr\left(\textstyle{\sum_{i=1}^bX_i}>\tfrac b4\right)
\to \frac 1 2.
\]
Thus for sufficiently large $\ell$,
\[
\Pr\left(\textstyle{\sum_{i=1}^b\widetilde X_i}>\tfrac b4 \mid  r[1:2]=g_1g_2\right)> \frac 13,
\]
and since each reduction shortens the word by at least $2k$ letters we have
\[
\Pr\left(\length(\bar r)<\ell - \tfrac{bk}2 \mid  r[1:2]=g_1g_2\right)> \frac 1 3.
\]

\begin{step}
Let $\bar R=\{\bar r\mid r\in R\}$ be the set of reduced words as above. 
For each pair of distinct elements $x,y$ chosen from the generators and their inverses,
\aas there exists a pair  $\bar r_1,\bar r_2\in\bar R$ such that $\bar r_1[1]=x$, $\bar r_2[1]=y$,
and
\[
\bar r_1[2:\length(\bar r_1) ]=\bar r_2[2:\length(\bar r_2)].
\]
Consequently,
$x=_G y$.  Triviality follows.
\end{step}

First, \eqref{spade} says that $b\to\infty$, so we have $b\geq 2$ for $\ell$ sufficiently large,
which gives
\[
\frac{bk}4-\ell f\geq \frac {k}2 -\ell f,
\]
and the right-hand side goes to infinity by \eqref{star}.

Next, let $x,y,z$ be chosen among the generators and their inverses such that $z^{-1}\neq x,y$
and $x\neq y$.  Recall that $R_w$ denotes words beginning with word $w$.
We examine relators $r\in R_{xz}\cup R_{yz}$ such
that $\length(\bar r)\leq\ell' = l-\frac{bk}{4}$.
Note that $|R_{xz}|$ is close to $\frac{|R|}{2m(2m-1)}$ a.a.s., 
and we expect $1/3$ of these to have enough reductions
so their length is no more than $\ell'$.
So we get
\[
|\{r\in R_{xz}\mid\length(\bar r)\leq\ell'\}|
>\frac {(2m-1)^{\ell(\frac 12 - f)}}{3(2m)(2m-1)+1},
\]
and the same holds for $R_{yz}$.
To apply Lemma~\ref{pigeons} to get matching tails, we must compare
the number of shortened words
to the square root of the number of possible tails.
(The two colors are initial 2-letter words and the 
boxes are final $(\ell'-2)$-letter words.)
In order to see that
\[
(2m-1)^{\ell(\frac 12 - f)}\gg \sqrt{(2m-1)^{\ell '-2}},
\]
note that
\[
\frac{(2m-1)^{\ell(\frac 12 - f)}}{\sqrt{(2m-1)^{\ell '-2}}}
\geq (2m-1)^{\tfrac{bk}4 - \ell f} \to\infty .
\]
We may conclude that  \aas
there exists a pair of words $r_1\in R_{xz}$ and $r_2\in R_{yz}$ such that
\[
\bar r_1[3:\length(\bar r_1)]=\bar r_2[3:\length(\bar r_2)],
\]
and since $\bar r_1=_G 1 =_G \bar r_2$, we get
$xz=_G yz$,
so finally $x=_G y$.
This means that \aas
all generators and their inverses are equal in $G$.
\end{proof}

\begin{proof}[Proof of Corollary~\ref{triviality-cor}]
For  \eqref{star} we compute
$k - 2 \ell f=\log\log \ell$,
which goes to infinity. % precisely when $c'>2c$.
Condition \eqref{spade} is equivalent to $b\to\infty$, and we calculate 
\begin{align*}
 \log b=\log \left( \frac{\ell - 2}{(2k+2)(2m-1)^{2k}} \right)
&= \log (\ell - 2)   - \log (2k + 2) - 2k \\
& \geq \log \ell - \log\log\ell -\log\ell+2\log\log\ell - C\\
& = \log\log\ell -C
\end{align*}
for a suitable constant $C$.
% This goes to infinity precisely when $c' < \frac 12$.
\end{proof}

%%%%%%%%%%
%%%%%%%%%%
\section{The hyperbolic range}

To prove hyperbolicity, we establish an isoperimetric inequality on reduced van Kampen diagrams
(RVKDs) for a random group, as in Ollivier~\cite[Chapter 5]{Ollivier05}.
The main difference to our argument is that, rather than aiming to show a linear isoperimetric
inequality directly, we show that the random group satisfies a quadratic isoperimetric
inequality with a small constant.
This in turn implies that the group is hyperbolic by a well-known result of Gromov
(see Papasoglu~\cite{Papasoglu} and Bowditch~\cite{Bowditch}).

Following Ollivier,
we write $D$ for a (reduced) van Kampen diagram; $|D|$ for its number of faces, and $|\partial D|$ for the length of its boundary.
(Note $|\partial D|\ge \#$ boundary edges because of possible ``filaments.")
A path of contiguous edges so that
all interior vertices have valence two is called a
{\em contour}.

The key fact which allows us to check the isoperimetric inequality only on diagrams
of certain sizes is the following theorem of Ollivier, which is a variation on 
Papasoglu's result in~\cite{Papasoglu}.

\begin{lem}[Local-global principle
\protect{\cite[Prop 9]{Ollivier07}}]\label{isoperimetric-inequality}
For fixed $\ell$ and $K\geq 10^{10}$, if
\[\underbrace{\tfrac{K^2}{4}\leq|D|\leq 480 K^2}_{\PP}
\Rightarrow\underbrace{|\partial D|^2\geq 2\cdot 10^4
\ell^2|D|}_{\QQ},\]
then
\[
|D|\geq K^2\Rightarrow|\partial D|
\geq \frac{\ell}{10^4 K}|D|.
\]
\end{lem}
That is, if RVKDs in a certain size range satisfy
a good enough {\em quadratic} isoperimetric inequality,
then all RVKDs satisfy a {\em linear} isoperimetric
inequality.  Later, we will let $K=K(\ell)$ to vary the window
of diagrams considered.

We will use Ollivier's definitions concerning {\em abstract
diagrams}, which are a device for precise bookkeeping
in van Kampen diagrams to control dependencies in
probabilities.  Roughly speaking, an abstract diagram is a van Kampen 
diagram where we forget the labelling of edges by generators and the labelling of 
faces by relators.  We do keep track 
of the orientation and starting point of the boundary of each face, 
and we also label faces so we know which faces bear the same relator.
(Since our relators are reduced but need not be cyclically reduced, each
face in an abstract diagram is allowed to have a single ``inward spur'', 
see~\cite[Page 83, footnote 4]{Ollivier05}.)

For our group to be \aas (infinite torsion-free)
hyperbolic, it suffices to have one RVKD for each
trivial word that satisfies the linear isoperimetric
inequality, so this statement for all RVKDs will be
more than enough.
To show
that \aas all diagrams satisfy the hypothesis,
we show that the probability of a diagram existing
that has \PP but not \QQ tends to $0$.  To calculate
this, we must first get a bound on how many abstract
diagrams have \PP, and the probability that such an
abstract diagram is fufillable from our relator set.

%%%
\subsection{Probability of fulfillability}
Still following Ollivier, we estimate the probability that some relators exist to 
fulfill $D$.
\begin{lem}[\protect{\cite[Lem 59]{Ollivier05}}]\label{fulfillability}
Let $R$ be a random set of relators with $|R|=\num(\ell)$
at length $\ell$.
Let $D$ be a reduced abstract diagram. Then we have
\[
\Pr(D \textrm{ is fulfillable})
\leq (2m-1)^{\textstyle{\frac 1 2
\left(\frac{|\partial D|}{|D|}- \ell +2 \log\num\right)} }
=(2m-1)^{\textstyle{\frac 1 2 \left( \frac{|\partial D| }{|D|} - \ell (1-2\D) \right)}}
\]
\end{lem}

In our case, our choice of $\num(\ell)$ gives $\D=\frac 12 - f(\ell)$.
If a diagram satisfies \PP and not \QQ, we get
$$\frac{| \partial D|}{|D|} < \frac{\sqrt{2} \sdot 10^2  \ell \sqrt{|D|}}{K^2/4}
< \frac{5 \sdot 10^3  \ell K}{K^2/4}=\frac{2\sdot 10^4\ell}{K}.$$
All together, we get
\[
\Pr(D \textrm{ is fulfillable})
\leq (2m-1)^{\textstyle{10^4 \frac{\ell}{K} - \ell\sdot f(\ell)} }.
\]

%%%
\subsection{Counting abstract diagrams}

There is a forgetful map from abstract diagrams $\Gamma$ to embedded planar graphs $\Gamma'$ that strips away the
data (i.e., subdivision of contours into edges, face labelings, and start points and orientations
for reading around each face).  Figure~\ref{abstract-diagram} shows an example.  To see that the planar
embedding matters, consider  the two different ways
of embedding a figure-eight---clearly
different as van Kampen diagrams.
(\raisebox{-.05in}{\tikz \draw (0,0) circle (.2) (.4,0) circle (.2);} versus \raisebox{-.05in}{\tikz  \draw (0,0) arc (-180:180:.22) (0,0) arc (-180:180:.17);} )
Adding data to a  graph to recover an abstract diagram will be called {\em filling in}.

In order to find an upper bound on the number of van Kampen diagrams up to a certain size,
we will count  possible abstract diagrams 
by enumerating planar
graphs and ways of filling in.

% ----- Picture: abstract diagram to embedded planar graph -----
\begin{figure}[ht]
\begin{tikzpicture}[scale=0.50]
%Label G and G'
\node at (4,3) {$\Gamma$};
\node at (17,3) {$\Gamma '$};

%Vertices for G
\node (v1) at (0,0) {};
\node (v2) at (-1,1) {};
\node (v3) at (-1,2.5) {};
\node (v4) at (0,3.5) {};
\node (v5) at (1.5,3.5) {};
\node (v6) at (2.5,2.5) {};
\node (v7) at (2.5,1) {};
\node (v8) at (1.5,0) {};
\node (v10) at (-1,-1) {};
\node (v11) at (-1,-2.5) {};
\node (v12) at (0,-3.5) {};
\node (v13) at (1.5,-3.5) {};
\node (v14) at (2.5,-2.5) {};
\node (v15) at (2.5,-1) {};
\node (v17) at (4,-3) {};
\node (v16) at (4,1.5) {};
\draw (0,0) node (v9) {};

%The boundaries of the faces for G
\draw  plot[smooth, tension=.7] coordinates {(v1) (v2) (v3) (v4) (v5) (v6) (v7) (v8) (0,0)};
\draw  plot[smooth, tension=.7] coordinates {(v9) (v10) (v11) (v12) (v13) (v14) (v15) (v8)};
\draw  plot[smooth, tension=.7] coordinates {(v7) (v16) (5.5,0.5) (5.5,-1.5) (v17) (v14)};

%Extra vertices on the top face of G
\draw (-1.1,2.6) -- (-0.9,2.4);
\draw (-0.1,3.6) -- (0,3.4);
\draw (1.5,3.6) -- (1.4,3.4);
\draw (-1.1,0.9) -- (-0.9,1);
\draw (2.4,2.4) -- (2.6,2.5);

%Extra vertices on the bottom face of G
\draw (-1.1,-1) -- (-0.9,-1.1);
\draw (-1.1,-2.6) -- (-0.9,-2.5);
\draw (-0.1,-3.6) -- (0,-3.4);
\draw (1.6,-3.6) -- (1.5,-3.4);
\draw (2.4,-1.1) -- (2.6,-1);

%Extra vertices on the right face of G
\draw (4,-2.9) -- (4.1,-3.1);
\draw (5.4,-1.5) -- (5.6,-1.6);
\draw (5.4,0.5) -- (5.6,0.6);
\draw (4,1.4) -- (4.1,1.6);

%Face labels of G
\node at (0.6,1.6) {1};
\node at (0.6,-2) {2};
\node at (4,-0.6) {1};

%Squiggle arrow
\draw [decorate,decoration=zigzag] (7,0) -- (9.8,0);
\draw [->] (9.8,0) -- (10,0);

%Orienations and start points of G
\draw [->] (-1.1,0.9) -- (-1.2,1.2);
\draw [->] (-1.1,-2.6) -- (-0.95,-2.9);
\draw [->] (4.1,-3.1) -- (4.4,-2.95);

%Vertices of G'
\node (r1) at (13,0) {};
\node (r2) at (12,1) {};
\node (r3) at (12,2.5) {};
\node (r4) at (13,3.5) {};
\node (r5) at (14.5,3.5) {};
\node (r6) at (15.5,2.5) {};
\node (r7) at (15.5,1) {};
\node (r8) at (14.5,0) {};
\node (r10) at (12,-1) {};
\node (r11) at (12,-2.5) {};
\node (r12) at (13,-3.5) {};
\node (r13) at (14.5,-3.5) {};
\node (r14) at (15.5,-2.5) {};
\node (r15) at (15.5,-1) {};
\node (r17) at (17,-3) {};
\node (r16) at (17,1.5) {};
\draw (13,0) node (r9) {};

%The boundaries of the faces on G'
\draw  plot[smooth, tension=.7] coordinates {(r1) (r2) (r3) (r4) (r5) (r6) (r7) (r8) (13,0)};
\draw  plot[smooth, tension=.7] coordinates {(r9) (r10) (r11) (r12) (r13) (r14) (r15) (r8)};
\draw  plot[smooth, tension=.7] coordinates {(r7) (r16) (18.5,0.5) (18.5,-1.5) (r17) (r14)};
\end{tikzpicture}
\caption{Abstract diagram and corresponding embedded planar graph.
\label{abstract-diagram}}
\end{figure}
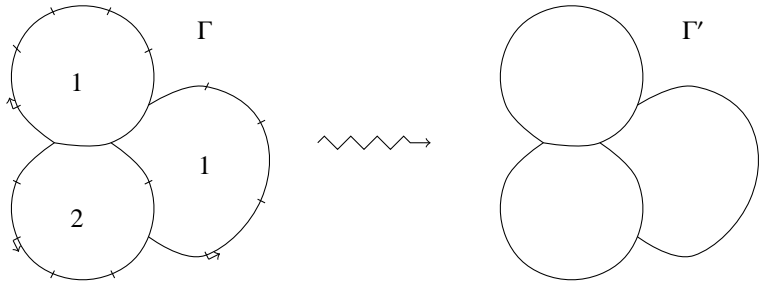
% --------------------------  End Picture ------------------------

\begin{prop}[Diagram count]
Let $N_{F}(\ell)$ be the number of abstract diagrams with at most $F$ faces, each of boundary length $\ell$.
Then $\log N_F(\ell)$ is  asymptotically bounded above by $6 F \log\ell +  2F \log F$.
\end{prop}

\begin{proof}
Consider abstract diagrams with no more than $F$ faces.
Since there are $\ell$ edges on the boundary of each face, two orientations,
and at most $F$ faces, there are no more than  $(2 \ell)^F$ choices of oriented start points.
Faces can have at most $F$ distinct labels, so there are at most $F^F$ possible labelings.

In order to estimate the number of ways we can subdivide the contours into edges,
we first count edges of $\Gamma '$.  If $\Gamma'$ has no inward spurs, then every vertex
has valence at least three.  Since the Euler characteristic is $V - E + F=1$,
we have $2E \geq 3V$, which simplifies to
$E\le 3F - 3 \le 3F$.  Each face of $\Gamma'$ can have at most one inward spur, which increases the number of edges
by $\leq 2$ for each face, so the total number of edges in $\Gamma'$ satisfies $E\le 5F$.

The number of ways to put $\ell$ edges around each face can be overcounted
by the number of ways to subdivide each contour into exactly $\ell$ edges,
which is $\ell^E$ and so is bounded above by $\ell^{5F}$.

Tutte shows in \cite[p.~254]{Tutte63} that
the number of embedded planar graphs with exactly $n$ edges is $\frac{2(2n)!3^n}{n!(n+2)!} $.
Using $E \leq 5F$, and $(n/e)^n \leq n! \leq n^n$ (with lower bound from Stirling's formula), we get
\begin{align*} \textrm{\# ($\Gamma'$ with $\le 5F$ edges)}
& \le\sum_{n=1}^{5F} \frac{2(2n)!\, 3^n}{n!\, (n+2)!} 
\leq 5F \frac{2(10F)!\, 3^{5F}}{(5F)!\,(5F+2)!} \\
& \leq \frac{(10F)! 3^{5F}}{(5F)!(5F)!} \leq
 \frac{(10F)^{10F} 3^{5F}}{(5F/e)^{10F}} = (2e)^{10F} 3^{5F} \leq 3^{25F}.
\end{align*}
Combining  the above information, we get
$$N_F(\ell)
\leq (2 \ell)^F F^F \ell^{5F} 3^{25F},$$
and so
% Using   $\log(n!) \le n \log n$, we get
$$\log N_F \leq F\log (2\ell)+ F\log F + 5F\log \ell + 25F\log 3.$$

Gathering terms of  highest order, we have an upper bound by $6 F \log\ell+2F \log F$, as claimed.
\end{proof}

\begin{cor} \label{diagram-bound}
Let $N^I(\ell)$ be the number of reduced van Kampen
diagrams with property \PP at relator length $\ell$.
Then $\log N^I(\ell)$ is  asymptotically bounded above by $3000 K^2 \log(K\ell)$.
\end{cor}
\begin{proof}
Considering all diagrams with  $|D|\le 480K^2$ will be an overcount, so we use $F=480K^2$ in
the above estimate, i.e., $N^I(\ell)\le N_{480K^2}(\ell)$.  
\end{proof}

\subsection{Hyperbolicity threshold}

\begin{thm}[Sufficient condition for hyperbolicity]\label{hyperbolicity} Given any $f(\ell)=o(1)$,
suppose there exists a function
$K: \N \to \N$ such that
\begin{equation}\tag{$*$}\label{asterisk}
3000 K^2 \log(K\ell) +  10^4 \tfrac{\ell}{K} - \ell\sdot f(\ell)   \to -\infty.
\end{equation}
Then $G\in\H(f)$ is \aas (infinite torsion-free) hyperbolic.
\end{thm}

\begin{remark}
In view of Corollary~\ref{diagram-bound},
one intuitive way of choosing a $K,f$ pair
is to take $K^2 \log \ell$ and $\ell f(\ell)$ to be of the
same order.
It turns out that 
we can do slightly better than that by instead choosing
to equalize the orders of $\frac{\ell}{K}$ and
$\ell f(\ell)$, which gives the  pair below.
\end{remark}

\begin{cor}
For any constants $c,c'$ with $0<4000c'^2+\frac{10^4}{c'}<c$,
the functions $f(\ell)=c  \frac{\log^{1/3} (\ell) }{\ell ^ {1/3}}$ and $K(\ell)=c' \frac{\ell ^ {1/3}}{\log^{1/3} (\ell) }$
satisfy \eqref{asterisk}.

In particular, for $c>10^5$, a random group in
$\H\left(c  \frac{\log^{1/3} (\ell) }{\ell ^ {1/3}}\right)$
is \aas (infinite torsion-free) hyperbolic.
\end{cor}

\begin{proof}[Proof of Theorem]
Observe that
\begin{align*}
P:=\Pr \left( \substack{\exists
\textrm{ a van Kampen diagram }D \\
\textrm{ that satisfies \PP but not \QQ}} \right)
& \leq \sum_{ \substack{\textrm{abstract diagrams } D \\ \textrm{ with \PP but not \QQ}}}
\Pr(D \textrm{ is fulfillable}) \\
& \leq  N^I(\ell) \cdot  (2m-1)^{ 10^4 \frac{\ell}{K} - \ell\sdot f(\ell)},
\end{align*}
where  $N^I(\ell)$ is as in Corollary~\ref{diagram-bound} 
% = \# \{\textrm{diagrams with \PP for a fixed length } \ell \} $
and $(2m-1)^{ 10^4 \frac{\ell}{K} - \ell\sdot f(\ell)}$
is the fulfillability bound from Lemma \ref{fulfillability}.
(Note that the last inequality vastly overcounts by replacing
$\left[\textrm{\PP and not \QQ}\right]$ with simply \PP.)

We will show that the  local-global principle
(Lemma \ref{isoperimetric-inequality})  holds \aas for all diagrams,
by showing that for a $K,f$ pair as in the hypothesis, the above quantities go to zero.
In particular, we will show that $\log P\to -\infty$.

We have
$ \log P
\leq  \log N^I + 10^4 \frac{\ell}{K} - \ell\sdot f(\ell).$
By applying
Corollary~\ref{diagram-bound}, we have this asymptotically bounded above by
$$3000 K^2 \log(K\ell) +  10^4 \tfrac{\ell}{K} - \ell\sdot f(\ell).$$
Requiring that this goes to $-\infty$ is exactly \eqref{asterisk}.
\end{proof}

\begin{proof}[Proof of Corollary]
We calculate each of the four terms of \eqref{asterisk} using
$K=c'\ell^{1/3}\log^{-1/3}\ell$ and $f=c\ell^{-1/3}\log^{1/3}\ell$.  We have
\[\begin{cases}
3000K^2\log(K\ell) \le 4000 c'^2 \ell^{2/3} \log^{1/3}\ell  \ ;\\
  10^4  \frac \ell K =  \frac{10^4}{c'} \ell^{2/3}  \log^{1/3}\ell\ ;\\
   \ell f = c  \ell^{2/3}  \log^{1/3} \ell.
\end{cases}\]
Provided $4000c'^2 + \frac{10^4}{c'} < c$, the expression goes to $-\infty$
and \eqref{asterisk} is verified.
For example, we can choose $c'=1$ and $c=10^5$.
\end{proof}

% It is also important to note that if we have a function that
% goes to $0$ slower than a function that yields hyperbolicity,
% then we have less possible van Kampen diagrams to violate
% our conditions, and thus the remaining half of Theorem
% \ref{bigtheorem} is proven.

%%%%
\subsection{Hyperbolicity constant}

\begin{thm}[Effective hyperbolicity constant]\label{hyperbolicity-constant}
Suppose $X$ is a 2--complex that is geometrically finite, i.e., there is some $N$ such that 
every face has at most $N$ edges.  
Suppose there is $\kappa>1/N$ so that $X$ has a linear isoperimetric inequality for large-area loops:
if an edge loop $\gamma$ in $X$ has area $\geq 18\kappa^2N^2$, 
then $\gamma$ can be filled with at most $\kappa n$ cells.
Then the one-skeleton of $X$ 
has $\delta$--thin triangles for $\delta=120 \kappa^2 N^3$.
\end{thm}

\begin{proof}
	We closely follow the proof 
	in \cite[III.H.2.9]{BridsonHaefliger} (replacing $K$ by $\kappa$ to avoid notation clash).
	If there is a triangle which is not $6k=18\kappa N^2$--thin, one builds
	a hexagon $\mathcal{H}$ (or quadrilateral) whose minimal area filling has area
	$\geq \kappa(\alpha-2k) \geq \kappa(6k) = 18\kappa^2N^2$.
	So this hexagon satisfies our linear isoperimetric hypothesis
	$\mathrm{|\mathcal{H}|} \leq \kappa |\partial \mathcal{H}|$.
	The remainder of the proof shows that the hexagon is $\delta$--thin provided
	\[
		\frac{\delta-3k}{3N} > 12 k\kappa \ \Longrightarrow
		 \delta > 3k+36kN\kappa = 9\kappa N^2 + 108\kappa^2N^3 \geq 117 \kappa^2 N^3.\qedhere
	\]
\end{proof}

\begin{cor}
Our density one-half random groups are hyperbolic with $\delta=c \ell^{5/3}$.
\end{cor}

By contrast, as noted above, the best known hyperbolicity constant for $d<1/2$ is proportional to $\ell$.  

\begin{proof}[Proof of Corollary]
The output of the local-to-global principle was the linear isoperimetric inequality 
$|\partial D| \ge  \frac{\ell}{10^4 K}|D|$ and to get the needed case we used
$K(\ell)= \frac{\ell ^ {1/3}}{\log^{2/3} (\ell) }$.
This gives $|D|\le c'' \ell^{-2/3}\log^{-2/3}(\ell) \sdot |\partial D| \le c'' \ell^{-2/3} |\partial D|$, so we take $\kappa = c'' \ell^{-2/3}$ and $N=\ell$.
This linear isoperimetric inequality holds for all diagrams $D$ of size
$|D| \geq K^2$; observe that 
$18\kappa^2 N^2 = 18 (c'' \ell^{-2/3})^2  \ell^2 \geq K^2$ for large $K$.
Therefore, Theorem~\ref{hyperbolicity-constant} gives that all triangles are $\delta$--thin
for a value of $\delta$ proportional to $\ell^{5/3}$.
\end{proof}

%%%%%%%%%%%

\end{document}